\documentclass[12pt,reqno]{amsart}
\usepackage{amsfonts}
\usepackage{amssymb}
\usepackage{amsmath}
\numberwithin{equation}{section}

\newcommand{\C}{{\mathbb C}}
\newcommand{\D}{{\mathbb D}}
\newcommand{\N}{{\mathbb N}}
\newcommand{\Z}{{\mathbb Z}}
\newcommand{\clos}{{\operatorname{clos}}}
\newcommand{\spn}{{\operatorname{span}}}
\newcommand{\dens}{{\operatorname{dens}}}
\newcommand{\eps}{\varepsilon}
\newcommand{\im}{\operatorname{Im}}
\newcommand{\re}{\operatorname{Re}}
\newcommand{\df}{\buildrel\mathrm{def}\over=}

\newtheorem{theorem}{Theorem}
\newtheorem{lemma}[theorem]{Lemma}
\newtheorem*{conjecture}{Conjecture}

\numberwithin{equation}{section}

\begin{document}

\title[Invertibility threshold for $H^\infty$ trace algebras]
{Invertibility threshold \\ for $H^\infty$ trace algebras, \\ and effective
matrix inversions}

\dedicatory{ Dedicated to the memory of M.~S.~Birman, \\
from whom both of us were learned a lot \textup(and not only
mathematics\textup)}

\author{Nikolai Nikolski}
\address{University Bordeaux1/V.A. Steklov. Math. Inst., St. Petersburg}
\email{Nikolai.Nikolski@math.u-bordeaux1.fr}

\author{Vasily Vasyunin}
\address{St. Petersburg Branch of the V.A. Steklov. Math. Inst.
\\
Fontanka 27, St. Petersburg, 191023, Russia} \email{vasyunin@pdmi.ras.ru}

\thanks{V. Vasyunin's research was supported in part by RFBR (grant 08-01-00723).}
\thanks{N. Nikolski's research was partially supported by the French ANR Projects DYNOP and FRAB}

%\subjclass[2000]{Primary chto syuda?}


\date{September 12, 2010}

\begin{abstract}
For a given $\delta$, $0<\delta<1$, a Blaschke sequence $\sigma=\{\lambda_j\}$
is constructed such that every function $f$, $f\in H^\infty$, having
$\delta<\delta_f=\inf_{\lambda\in\sigma}|f(\lambda)|\le\|f\|_\infty\le1$ is
invertible in the trace algebra $H^\infty|\sigma$ (with a norm estimate of the
inverse depending on $\delta_f$ only), but there exists $f$ with
$\delta=\delta_f\le\|f\|_\infty\le1$, which does not. As an application, a
counterexample to a stronger form of the Bourgain--Tzafriri restricted
invertibility conjecture for bounded operators is exhibited, where an
``orthogonal (or unconditional) basis'' is replaced by a ``summation block
orthogonal basis''.
\end{abstract}
\maketitle

\section{Introduction}

The paper deals with a numerical control
of inverses (condition numbers) for functions $T=f(A)$ of
large matrices in terms of the lower spectral parameter
$$
\delta =\delta(T)=\min|\lambda_j(T)|
$$
Precisely, our problem is the following. Given a sequence
$\sigma=\{\lambda_j\}$ in the unit disk $\D=\{z\in\C\colon|z|<1\}$ of the
complex plain, we consider all normalized matrices $A$, $\|A\|\le1$ (or Hilbert
space operators) such that $\sigma(A)\subset\sigma$ (counting multiplicities)
and look for a numerical function $c(\delta)=c(\delta,\sigma)$ bounding the
inverses
$$
\|\,T^{-1}\|\le c(\delta)
$$
for all $T=f(A)$ having $\delta\le|\lambda_j(T)|\le\|T\|\le1$, where
$\lambda_j(T)$ mean eigenvalues of $T=f(A)$. The best possible upper bound
$c(\delta)$ is called $c_1(\delta)=c_1(\delta,\sigma)$,
\begin{align*}
c_1&(\delta,\sigma)=
\\
&\sup\big\{\|T^{-1}\|\colon T=f(A), \delta\le|\lambda_j(T)|\le\|T\|\le1,\
\sigma(A)\subset\sigma,\ \|A\|\le1\big\}\,.
\end{align*}
Here $f$ can be a polynomial (if $A$ is a finite matrix) or an $H^\infty$
function (if $A$ is a Hilbert space contraction). Recall that
$$
H^\infty=\left\{f\colon f\text{ holomorphic on }\D\text{ and }
\|f\|_\infty\!=\sup_{z\in\D}|f(z)|<\infty\right\}\,.
$$

Since $\delta\mapsto c_1(\delta,\sigma)$, $0<\delta<1$, is a decreasing
function, we can define a {\it critical constant} (or, an {\it invertibility
threshold}) $\delta_1=\delta_1(\sigma)$, $0\le\delta_1\le1$, by the following
properties
\begin{align*}
0<\delta<\delta_1&\implies c_1(\delta)=\infty\,,
\\
\delta_1<\delta\le1&\implies c_1(\delta)<\infty\,.
\end{align*}
The number $\delta_1$ can be considered as a threshold of bounded invertibility
or as a threshold for an operator algebra to be {\it inverse closed\/}:
operators $T$ from our collection with a ``{\it scattered\/}'' spectral data
(i.~e., $\inf_j|\lambda_j(T)|<\delta_1$, $\|T\|=1$) are, in general, not
invertible, whereas those with ``{\it flat\/}\,'' spectral data
$\delta_1<\delta\le|\lambda_j(T)|\le\|T\|\le1$ are invertible.

The principal result of this paper is a construction of a Blaschke sequence
$\sigma$ with a given in advance value of the critical constant
$\delta_1(\sigma)=\delta_1$, $0\le\delta_1\le1$ (Section 2 below).

The case, where $\delta_1=0$, was considered in~[GMN]; moreover, the paper
quoted contains necessary and sufficient conditions for $\delta_1(\sigma)=0$,
which reduces to the so-called Weak (Carleson) Embedding Property (WEP). See
the statement of the result at the end of this Introduction.

It is worth mentioning that, strictly speaking, the properties of a function
algebra $A$ on a set $\sigma$ to be inverse closed (i.~e., the property $f\in
A$, $\inf_{z\in\sigma}|f(z)|>0\implies1/f\in A$) does not imply that
$\delta_1(\sigma,A)=0$ (this fact was already mentioned in~\cite{GMN}; the
constants $c_1(\sigma,A)$ and $\delta_1(\sigma,A)$ are defined for an algebra
$A$ in a similar way). Indeed, for an arbitrary Blaschke sequence
$\sigma=\{\lambda_j\}$, the trace algebra $A=C_a(\D)|\sigma$ of the disk
algebra $C_a(\D)= H^\infty\cap C(\overline{\D})$, is always inverse closed,
whereas $c_1(\delta,C_a(\D)|\sigma)=c_1(\delta,H^\infty|\sigma)$ for every
$\delta$, $0<\delta<1$, and hence $\delta_1(\sigma,C_a(\D)|\sigma)=
\delta_1(\sigma,H^\infty|\sigma)$, but the algebra $H^\infty|\sigma$ can be not
inverse closed (i.~e. possibly $\delta_1(\sigma,C_a(\D)|\sigma)>0$). These
properties are shown in~\cite{GMN}.

The constant $c_1(\delta,\sigma)$ has a meaning of ``the best estimate for the
worst case'' when bounding inverse matrices in terms of the lower spectral
parameter $\delta$. Moreover, we can describe it in two more ways, at least in
the case of the ``simple spectrum'' (the points $\lambda_j$ of the sequence
$\sigma$ are pairwise different). First, it is the optimal upper bound for
inverses in the {\it trace algebra}
$$
H^\infty|\sigma =\left\{a\colon\sigma\to\C\colon\exists f\in H^\infty \text{
such that }a=f|\sigma\right\}
$$
endowed with the trace norm $\|a\|=\inf\{\|f\|_\infty\colon a=f|\sigma\}$.
Examples of such algebras with a given threshold $\delta_1$ of the bounded
invertibility (Section 2 below) are, probably, of interest for the $ H^\infty$
interpolation theory.

Secondly, in the definition of $c_1(\delta,\sigma)$, we can restrict ourselves
to a just one (``the worst'') contraction $A$ and the algebra generated by
$H^\infty$ functions of it. This is the so-called {\it model contraction}
$M_B$, which can be defined as follows. Given a Blaschke product $B=B_\sigma$
$$
B_\sigma=\prod_{j\geq 1}b_{\lambda_j}\,,
$$
where $b_\lambda=\frac{\lambda-z}{1-\overline\lambda z}\cdot\frac{|\lambda|}
\lambda$, $\lambda\in\D$, and $\sigma=\{\lambda_j\}$, $\sum_j(1-|\lambda_j|)<
\infty$ (the Blaschke condition), we set
$$
M^*_Bf=\frac{f-f(0)}z\,,\qquad f\in K_B\,,
$$
where $K_B=H^2\ominus BH^2$ (the orthogonal complement of $BH^2$ in $H^2$) and
$H^2$ stands for the standard Hardy space of the disk,
$$
H^2=\Big\{f=\sum_{k\ge0}a_kz^k\colon\sum_{k\ge0}|a_{k}|^2=
\|f\|_2^2<\infty\Big\}\,.
$$
It is well known (and easy to verify, see~\cite{Nik1}, \cite{Nik2}) that
\begin{gather*}
M_BK_B\subset K_B,
\\
\|M_B\|=1,
\\
\text{and}\quad \sigma(M_B)=\clos\{\lambda_j\colon j=1,2,\dots\}.
\end{gather*}
Moreover, $\|M^{-1}_B\|=1/B(0)$. It is also known that for every matrix $A$
with $\|A\|\le1$ and $\sigma(A)\subset\sigma$, one has $\|f(A)\|\le\| f(M_B)\|$
for every function $f$. This entails that the question on the invertibility and
the norm control of inverses can be reduced to functions $f(M_B)$ of the model
operator only, and inverse $f(M_B)^{-1}$, if it exists, is again a
$H^\infty$-function of $M_B$.

The above discussion easily implies the following.

\par (1) If the set $\{\lambda\in\sigma\colon\delta\le|\lambda|\le\delta'\}$ is
infinite for some $0<\delta\le\delta '<1$, then $c_1(\delta,\sigma)=\infty$.
\par (2) If $\sigma$ is a sequence tending to the unit circle (i.~e.,
$\{\lambda\in\sigma\colon|\lambda|\le\delta\}$ is finite for every $\delta<1$)
and $\sum_{\lambda\in\sigma}(1-|\lambda|)=\infty$, then
$c_1(\delta,\sigma)=\infty$ for every $\delta $, $0<\delta<1$.

These properties show that, in fact, {\it the Blaschke condition is necessary
in order our questions} (to find or to estimate $c_1(\delta,\sigma)$ and
$\delta_1(\sigma)$) {\it to be nontrivial}. In what follows we always assume
this property (if the converse does not stated explicitly). Now, we can give
the following expression for $c_1(\delta,\sigma)$.

\begin{lemma}
\label{1.1} Let $\sigma$ be a Blaschke subset of the unit disk $\D$. Then
$$
c_1(\delta,\sigma)=c_1(\delta,H^\infty|\sigma)=c_1(\delta,H^\infty\!/BH^\infty)
$$
for every $\delta$, $0<\delta<1$, where $B=B_\sigma$ and
\begin{align*}
c_1(\delta,H^\infty\!/BH^\infty)&=:\sup\Big\{\big\|\frac1f\big\|_{H^\infty\!/BH^\infty}
\colon\|f\|_\infty\le1,\delta\le\|f(\lambda)\|\text{ for }
\lambda\in\sigma\Big\}
\\
&=\sup\Big\{\inf\big[\|g\|_\infty\colon gf+hB=1\big]\colon
\\
&\qquad\qquad\qquad\qquad\delta\le\|f(\lambda)\|\le\|f\|_\infty\le1\text{ for }
\lambda\in\sigma\Big\}
\end{align*}
and $\|h\|_{H^\infty\!/BH^\infty }$ means $\inf\{\|g\|_\infty\colon g(\lambda)=
h(\lambda)\text{ for }\lambda\in\sigma\}$.
\end{lemma}

\begin{proof}
For every matrix $A$, $\|A\|\le1$, and $f\in H^\infty$, the von Neumann
inequality entails
$$
\|f(A)\|\le\|f\|_\infty\,.
$$
Since $B(A)=0$, $B=B_\sigma$, for $A$ having $\sigma(A)\subset\sigma$, we get
$$
\|f(A)\|\le\inf_{g\in H^\infty}\|f+Bg\|_\infty=\|f\|_{H^\infty\!/BH^\infty}\,.
$$
This implies $f(A)^{-1}=h(A)$ and $\|f(A)^{-1}\|\le\|h\|_{H^\infty}$ for every
solution $h$ of the equation $fh+Bk=1$, and therefore
$\|f(A)^{-1}\|\le\|\frac1f\|_{H^\infty\!/BH^\infty}$. Thus,
$c_1(\delta,\sigma)\le c_1(\delta,H^\infty\!/BH^\infty)$.

On the other hand, there exists an ``extreme operator'' (matrix) for which the
above calculus inequality becomes an identity. Indeed, if $A=M_B$, the ``model
operator'' mentioned above, then $\|h(M_B)\|=\|h\|_{H^\infty\!/BH^\infty}$ for
every $h\in H^\infty$ (Sarason's commutant lifting theorem, see for
example,~\cite{Nik1} or~\cite{Nik2}). Hence, $c_1(\delta,\sigma)\ge
c_1(\delta,H^\infty\!/BH^\infty)$.
\end{proof}

Finally, we quote the principal result from~\cite{GMN}

\begin{theorem}\textup(\cite{GMN}\textup)
Let $\sigma=\{\lambda_j\}$ be a Blaschke sequence in the disk $\D$. The
following are equivalent.
\begin{itemize}

\item[(1)] {$\delta_1(H^\infty|\sigma)=0$.}

\item[(2)] {The following Weak Embedding Property holds\textup: for every
$\eps>0$ there exists $C$ such that
$$
\sum_{j\ge1}\frac{(1-|\lambda_j|^2)(1-|z|^2)}{|1-\bar\lambda_jz|^2}\le C
$$
for every
$z\in\D\setminus\bigcup_{\lambda\in\sigma}\big\{\zeta\colon|b_\lambda(\zeta)|
<\eps\big\}$\,.}

\item[(3)] {For every $\eps>0$ there exists $\eta$ such that $|B(z)|\le\eta$
implies $\inf_{\lambda\in\sigma}|b_\lambda(z)|\le\eps$\textup; here $B$ is the
corresponding Blaschke product $B=\prod_{\lambda\in\sigma}b_\lambda$\,.}
\end{itemize}
Moreover, if $\eta(\eps)=\max\{\eta\}$ over all $\eta$ admitted in $(3)$, then
$$
\frac1{\eta(\delta)}\le c_1(\delta,H^\infty|\sigma)\le\frac a{\eta(\delta/3)^2}
\log\frac1{\eta(\delta/3)}
$$
for every $\delta,$ $0<\delta<1;$ $a>0$ is a numerical constant.
\end{theorem}

The paper is organized as follows. Section~2 contains our principal result: for
a given $\delta$, $0<\delta<1$, there exists a Blaschke product $B=B_\sigma$
such that $\delta_1(\sigma,H^\infty|\sigma )=\delta$. We also exhibit an upper
estimate for $c_1(\delta,H^\infty|\sigma)$ for $\delta_1<\delta\le1$. Since the
problem (and our result) on the invertibility threshold is conformally
invariant, we will change the variable and work (in Section~2) in the upper
half-plane $\C_+=\{z\in\C\colon\im(z)>0\}$ instead of the unit disk $\D$.

In Section~3, we use the above result in order to give a counterexample to a
stronger form of the so-called (Bourgain--Tzafriri) restricted invertibility
conjecture. The conjecture claims (see~\cite{CCLV} and comments in Section~3):
for every unconditional normalized basic sequence $\{x_j\}_{j\in J}$ in a
Hilbert space $H$ and for every bounded operator $T\colon H\to H$ having
$\inf_{j\in J} \|Tx_j\|>0$ there exists a partition $J=\bigcup_{i=1}^r\!J_i$
such that all restrictions $T|H_{J_i}$, $i=1,\ldots,r$, are left invertible;
here $H_{J'}= \spn\{x_j\colon j\in J'\}$ for every $J'\subset J$. The
conjecture is still open (June 2010). A stronger form (which is disproved in
Section~3) claims the same property but for all summation basic sequences
$\{x_j\}$.

\section{Algebras $H^\infty|\sigma$ with a given constant $\delta_1$}

We start with some geometrical considerations. In this section the symbol
$b_\lambda$ always means the Blaschke factor with the zero $\lambda$ in the
upper half-plane, i.~e.,
$$
b_\lambda(z)=\frac{z-\lambda}{z-\overline\lambda}\cdot\frac{|1+\lambda^2|}{1+\lambda^2}\,.
$$

\begin{lemma}
\label{l1} The rectangle
$$
\left\{z\colon\frac{\sqrt{1+\eps^2}}{\sqrt{1+\eps^2}+\sqrt2\,\eps}\le \frac{\im
z}{\im\lambda}\le\frac{\sqrt{1+\eps^2}}{\sqrt{1+\eps^2}-\sqrt2\,\eps},\
\frac{|\re(z-\lambda)|}{\im\lambda}\le\frac{\sqrt2\,\eps}{\sqrt{1-\eps^2}}\right\}
$$
is inscribed into the circle $\{z\colon|b_\lambda(z)|\le\eps\}$.
\end{lemma}

\begin{proof}
Put
$$
a=\frac{\re(z-\lambda)}{\im\lambda}\qquad\text{and}\qquad b=\frac{\im
z}{\im\lambda},
$$
then
$$
|b_\lambda(z)|^2=\left|\frac{z-\lambda}{z-\bar\lambda}\right|^2=\frac{a^2+(b-1)^2}{a^2+(b+1)^2}.
$$
We have to check that the vertices of the rectangle are on the mentioned
circle, i.~e., we need to check that the equality
$$
\frac{a^2+(b-1)^2}{a^2+(b+1)^2}=\eps^2
$$
holds if
$$
a=\pm\frac{\sqrt2\,\eps}{\sqrt{1-\eps^2}},\qquad
b=\frac{\sqrt{1+\eps^2}}{\sqrt{1+\eps^2}\pm\sqrt2\,\eps}.
$$
We shall verify the required identity in the form
$(a^2+b^2+1)(1-\eps^2)=2b(1+\eps^2)$:
\begin{align*}
(a^2+b^2+1)(1-\eps^2)=&\left(\frac{2\eps^2}{1-\eps^2}+
\frac{1+\eps^2}{(\sqrt{1+\eps^2}\pm\sqrt2\,\eps)^2}+1\right)(1-\eps^2)=
\\
&2\eps^2+(1+\eps^2)\frac{\sqrt{1+\eps^2}\mp\sqrt2\,\eps}{\sqrt{1+\eps^2}\pm\sqrt2\,\eps}
+1-\eps^2=
\\
&(1+\eps^2)\frac{2\sqrt{1+\eps^2}}{\sqrt{1+\eps^2}\pm\sqrt2\,\eps}=2b(1+\eps^2).
\end{align*}
\end{proof}

Now, we are using Frostman shifts of an inner function $\Theta$:
$$
\Theta_c\df\frac{\Theta+c}{1+\bar c\Theta}\,,
$$
which is known to be a Blaschke product for almost all values $c$, $|c|<1$. In
some cases it is easy to check that this is a Blaschke product {\it for all\/}
$c\ne0$. For example, this is the case for $\Theta=e^{iaz}$, $a>0$. Indeed, the
inner function $\Theta_c$ is analytic in a neighborhood of any real point,
therefore it could have a singular factor with a mass at infinity only. But
there is no such factor because $\lim_{y\to+\infty}\Theta_c(iy)=c\ne0$. For
more details, see, for example,~\cite{Gar} or~\cite{Nik1}.

\begin{lemma}
\label{l2} Let $z_k$ be zeroes of the Blaschke product
$$
B_{\alpha,\gamma}=\frac{e^{\pi i\gamma z}+e^{-\pi\alpha}} {1+e^{\pi(i\gamma
z-\alpha)}}=\prod_{k=-\infty}^\infty b_{z_k}\,,
$$
i.~e.\textup, $z_k=(2k+1+i\alpha)/\gamma$, $z_k\in\Z$. Then the strip
$$
S_{\alpha,\gamma}=\left\{z\colon\frac{\alpha\sqrt{1+\alpha^2}}{\sqrt{1+\alpha^2}+1}
\le\gamma\im z \le\frac{\alpha\sqrt{1+\alpha^2}}{\sqrt{1+\alpha^2}-1}\right\}
$$
is in the set
$$
\bigcup_{k=-\infty}^\infty\{z\colon|b_{z_k}(z)|<\eps\},
$$
if $\eps>1/\sqrt{1+2\alpha^2}$.
\end{lemma}

\begin{proof}
Apply Lemma~\ref{l1} with $\lambda=z_k$ and $\eps=1/\sqrt{1+2\alpha^2}$. Then
the sides of the rectangle are
$$
\frac{\sqrt{1+\eps^2}}{\sqrt{1+\eps^2}\pm\sqrt2\,\eps}=
\frac{\sqrt{1+\alpha^2}}{\sqrt{1+\alpha^2}\pm1} \qquad\text{and}\qquad
\frac{\sqrt2\,\eps}{\sqrt{1-\eps^2}}=\frac1\alpha,
$$
i.~e., the rectangle from Lemma~\ref{l1} is
$$
\left\{z\colon\frac{\sqrt{1+\alpha^2}}{\sqrt{1+\alpha^2}+1}\le
\frac\gamma\alpha\im z\le\frac{\sqrt{1+\alpha^2}}{\sqrt{1+\alpha^2}-1}, \
|\re(z-z_k)|\le\frac1\gamma\right\}.
$$
It is clear that the union of these rectangles gives just the required strip.
\end{proof}

{\bf Remark 1.} Let us note that the set
$$
S_{\alpha,\gamma}\setminus
\bigcup_{k=-\infty}^\infty\left\{z\colon|b_{z_k}(z)|<\frac1{\sqrt{1+2\alpha^2}}\right\}
$$
consists of a discrete set of points
$$
\frac1\gamma\left(2m+\frac{\alpha\sqrt{1+\alpha^2}}{\sqrt{1+\alpha^2}\pm1}\,i\right)
$$
on the upper and lower boundaries of the strip $S_{\alpha,\gamma}$ and the
distance from any such point to the set of zeroes $\{z_k\}$ is equal to
$1/\sqrt{1+2\alpha^2}$.

\begin{lemma}
\label{l3} The Blaschke product $B=\prod_{n=0}^\infty B_{\alpha,\beta^n\!\rho}$
converges for all $\alpha,\beta,\rho$ such that $\alpha>0$\textup,
$0<\beta<1$\textup, $\rho>0$. If
$$
\beta=\frac{\sqrt{1+\alpha^2}-1}{\sqrt{1+\alpha^2}+1}\,,
$$
then the half-plane
$$
\Pi_{\alpha,\rho}=\left\{z\colon\im z\ge
\frac{\alpha\sqrt{1+\alpha^2}}{\rho(\sqrt{1+\alpha^2}+1)}\right\}
$$
is in the set
$$
\bigcup_{\lambda\in \sigma(B)}\{z\colon|b_\lambda(z)|<\eps\}
$$
for $\eps>1/\sqrt{1+2\alpha^2}$.
\end{lemma}

\begin{figure}
\begin{center}
\begin{picture}(360,200)
\thinlines
\put(160,0){\vector(0,1){180}}
\put(10,20){\vector(1,0){300}}
\linethickness{.3pt}
\put(40,60){\circle*{2}}
\put(56,60){\circle*{2}}
\put(72,60){\circle*{2}}
\put(88,60){\circle*{2}}
\put(104,60){\circle*{2}}
\put(120,60){\circle*{2}}
\put(136,60){\circle*{2}}
\put(152,60){\circle*{2}}
\put(168,60){\circle*{2}}
\put(184,60){\circle*{2}}
\put(200,60){\circle*{2}}
\put(216,60){\circle*{2}}
\put(232,60){\circle*{2}}
\put(248,60){\circle*{2}}
\put(264,60){\circle*{2}}
%\put(280,60){\circle*{2}}
\put(280,60){\footnotesize $z_k=\frac{2k+1+i\alpha}{\rho}$}
\put(28,80){\circle*{2}}
\put(52,80){\circle*{2}}
\put(76,80){\circle*{2}}
\put(100,80){\circle*{2}}
\put(124,80){\circle*{2}}
\put(148,80){\circle*{2}}
\put(172,80){\circle*{2}}
\put(196,80){\circle*{2}}
\put(220,80){\circle*{2}}
\put(244,80){\circle*{2}}
\put(268,80){\circle*{2}}
%\put(292,80){\circle*{2}}
\put(280,80){\footnotesize $z_k=\frac{2k+1+i\alpha}{\rho\beta}$}
\put(34,110){\circle*{2}}
\put(70,110){\circle*{2}}
\put(106,110){\circle*{2}}
\put(142,110){\circle*{2}}
\put(178,110){\circle*{2}}
\put(214,110){\circle*{2}}
\put(250,110){\circle*{2}}
%\put(286,110){\circle*{2}}
\put(280,110){\footnotesize $z_k=\frac{2k+1+i\alpha}{\rho\beta^2}$}
\put(25,155){\circle*{2}}
\put(79,155){\circle*{2}}
\put(133,155){\circle*{2}}
\put(187,155){\circle*{2}}
\put(241,155){\circle*{2}}
%\put(295,155){\circle*{2}}
\put(280,155){\footnotesize $z_k=\frac{2k+1+i\alpha}{\rho\beta^3}$}
\put(152,62){\circle{22}}
\put(168,62){\circle{22}}
\put(148,84){\circle{34}}
\put(172,84){\circle{34}}
\end{picture}
\caption{}
\label{zeroes}
\end{center}
\end{figure}
(Fig.~\ref{zeroes} illustrates zeroes of $B$ and four circles
$|b_\lambda(z)|=\frac1{\sqrt{1+2\alpha^2}}$ for zeroes
$\lambda=\frac{\pm1+i\alpha}\rho$ and $\lambda=\frac{\pm1+i\alpha}{\beta\rho}$)

\begin{proof}
The following estimate implies convergence of $B$:
\begin{align*}
1-B_{\alpha,\beta^n\!\rho}(i)=&1-\frac{e^{-\pi\beta^n\!\rho}+e^{-\pi\alpha}}
{1+e^{-\pi(\alpha+\beta^n\!\rho)}}=
\\
&\frac{(1-e^{-\pi\beta^n\!\rho})(1-e^{-\pi\alpha})}
{1+e^{-\pi(\alpha+\beta^n\!\rho)}}\le
\\
&(1-e^{-\pi\alpha})\pi\beta^n\!\rho.
\end{align*}
It remains to note that for $\beta=(\sqrt{1+\alpha^2}-1)/(\sqrt{1+\alpha^2}+1)$
and $\gamma_n=\beta^n\!\rho$ the upper boundary of the strip from
Lemma~\ref{l2} for $\gamma=\gamma_{n-1}$ coincides with the lower boundary of
the strip for $\gamma=\gamma_n$. Therefore the union of these strips gives just
the required half-plane.
\end{proof}

{\bf Remark 2.} The set of points on the imaginary axis
$$
v_n=\frac1{\rho\beta^n}\,\frac{\alpha\sqrt{1+\alpha^2}}{\sqrt{1+\alpha^2}+1}\,i
$$
(the points of intersection of four corresponding circles as
on~Fig.\ref{zeroes}) is included into
$$
\Pi_{\alpha,\rho}\setminus \bigcup_{\lambda\in
\sigma(B)}^\infty\left\{z\colon|b_{\lambda}(z)|<\frac1{\sqrt{1+2\alpha^2}}\right\}.
$$
Every point $v_n$ has four nearest zeroes of $B$, namely,
$(i\alpha\pm1)/(\rho\beta^n)$ and $(i\alpha\pm1)/(\rho\beta^{n-1})$ with the
pseudohyperbolic distance just $1/\sqrt{1+2\alpha^2}$ from each of them.

Recall that pseudohyperbolic distance between two points $z,w\in\C_+$ is
defined by
$$
|b_w(z)|=\left|\frac{z-w}{z-\bar w}\right|
$$
and between two points $z,w\in\D$:
$$
\quad |b_w(z)|=\left|\frac{z-w}{1-\bar wz}\right|\,.
$$

\begin{lemma}
\label{l4} For the Blaschke product $B=\prod_{n=0}^\infty
B_{\alpha,\beta^n\!\rho}$ a lower estimate
$$
|B(x+iy)|\ge\exp\left\{-\frac{\left(1+e^{-\pi\alpha}\right)\pi\rho y}
{\left(e^{-\pi\rho y} -e^{-\pi\alpha}\right)(1-\beta)}\right\}
$$
is true in the strip $0<y<\frac\alpha\rho$. In the complementary half-plane
$y>\frac\alpha\rho$ we have the following upper estimate
$$
|B(x+iy)|\le\exp\left\{-\frac{\log(\cosh\pi\alpha)\cdot\log\frac{\rho
y}{\alpha}}{\log\frac1\beta}\right\}\,.
$$
\end{lemma}

\begin{proof}
For the product $B_{\alpha,\gamma}$ we have
$$
|B_{\alpha,\gamma}(x+iy)|^2=\frac{e^{-2\pi\gamma y} +e^{-2\pi\alpha}
+2e^{-\pi(\gamma y+\alpha)}\cos\pi\gamma x} {1+e^{-2\pi(\gamma y+\alpha)}
+2e^{-\pi(\gamma y+\alpha)}\cos\pi\gamma x}\,
$$
and therefore
$$
\Big|\frac{e^{-\pi\gamma y} -e^{-\pi\alpha}}{1-e^{-\pi(\gamma y+\alpha)}}\Big|
\le|B_{\alpha,\gamma}(x+iy)|\le \frac{e^{-\pi\gamma y}
+e^{-\pi\alpha}}{1+e^{-\pi(\gamma y+\alpha)}}\,.
$$
Now, we deduce an estimate from below assuming $0<y<\frac\alpha\rho$\,:
\begin{align*}
\log&\frac1{|B(x+iy)|}=\sum_{n=0}^\infty\log\frac1{|B_{\alpha,\beta^n\!\rho}(x+iy)|}
\le\sum_{n=0}^\infty\log\frac{1-e^{-\pi(\beta^n\!\rho
y+\alpha)}}{e^{-\pi\beta^n\!\rho y} -e^{-\pi\alpha}}
\\
&\le\sum_{n=0}^\infty\left(\frac{1-e^{-\pi(\beta^n\!\rho y+\alpha)}}
{e^{-\pi\beta^n\!\rho y} -e^{-\pi\alpha}}-1\right)=
\sum_{n=0}^\infty\frac{\left(1+e^{-\pi\alpha}\right)\left(1-e^{-\pi\beta^n\!\rho
y}\right)}{e^{-\pi\beta^n\!\rho y} -e^{-\pi\alpha}}
\\
&\le\left(1+e^{-\pi\alpha}\right)\sum_{n=0}^\infty\frac{\pi\beta^n\!\rho y}
{e^{-\pi\rho y} -e^{-\pi\alpha}}=\frac{\left(1+e^{-\pi\alpha}\right)\pi\rho y}
{\left(e^{-\pi\rho y} -e^{-\pi\alpha}\right)(1-\beta)}\,,
\end{align*}
as it was claimed. To estimate $|B(x+iy)|$ from above we replace $B$ by a
finite product $\prod_{0\le n\le N}B_{\alpha,\beta^n\!\rho}$, where
$$
N\df\frac{\log\frac{\rho y}\alpha}{\log\frac1\beta}\,.
$$
The number of such indices $n$ is $[N]+1>N$. Since for these $n$ we have
$$
\beta^n\!\rho y\ge\alpha\,,
$$
for each factor we get an estimate
$$
|B_{\alpha,\beta^n\!\rho}(x+iy)|\le \frac{e^{-\pi\beta^n\!\rho y}
+e^{-\pi\alpha}}{1+e^{-\pi(\beta^n\!\rho y+\alpha)}}\le
\frac{2e^{-\pi\alpha}}{1+e^{-2\pi\alpha}}=\frac1{\cosh\pi\alpha}\,.
$$
Therefore for the whole product we have
$$
|B(x+iy)|\le(\cosh\pi\alpha)^{-N}\,.
$$
\end{proof}

From now on, we fix
$$
\beta=\frac{\sqrt{1+\alpha^2}-1}{\sqrt{1+\alpha^2}+1}\,,
$$
and consider $B=\prod_{n=0}^\infty B_{\alpha,\beta^n\!\rho}$ corresponding to
this $\beta$.

\begin{theorem}
\label{main}
$$
\delta_1(H^\infty/BH^\infty)=\frac1{\sqrt{1+2\alpha^2}}\,.
$$
Moreover\textup, there exists an absolute constant $c$ and another constant
$C=C(\delta_1)$ such that
$$
c_1(\delta)\le\max\big\{\frac{c}{(\delta-\delta_1)^2}\log\frac1{\delta-\delta_1},
\;C\big\}
$$
for every $\delta$, $\delta_1<\delta\le1$.
\end{theorem}

The proof of the Theorem is contained in two following lemmata, where
$\delta_1$ means simply the number $\frac1{\sqrt{1+2\alpha^2}}$. After proving
these lemmata we can conclude that $\delta_1=\delta_1(H^\infty/BH^\infty)$.

\begin{lemma}
\label{l5} Let $\delta>\delta_1$. Then
$$
c_1(\delta)\le\max\big\{\frac{c}{(\delta-\delta_1)^2}\log\frac1{\delta-\delta_1},
\;C\big\}
$$
for some an absolute constant $c$ and another constant $C=C(\delta_1)$.
\end{lemma}

\begin{proof}
First we check that the function $|f(z)|+|B(z)|$ can be separated from zero by
some constant $\eta$ depending on $\alpha$ and $\delta$ only. By Lemma~\ref{l4}
in the strip
$$
0<y\rho\le\frac{\alpha\sqrt{1+\alpha^2}}{\sqrt{1+\alpha^2}+1}
$$
we have the estimate
$$
|B(x+iy)|\ge\exp\left\{-\frac{\alpha\sqrt{1+\alpha^2}\left(e^{\pi\alpha}+1\right)}
{2\left(e^{\frac{\pi\alpha}{\sqrt{1+\alpha^2}+1}}+1\right)}\right\}\,.
$$
Now we check that $|f(z)|$ is separated from zero in the half-plane
$$
y\rho\ge\frac{\alpha\sqrt{1+\alpha^2}}{\sqrt{1+\alpha^2}+1}\,.
$$
Fix any $\eps$, $\delta>\eps>\delta_1$. By Lemma~\ref{l3}
$$
\left\{z\colon\im
z\ge\frac{\alpha\sqrt{1+\alpha^2}}{\rho\big(\sqrt{1+\alpha^2}+1\big)}\right\}\subset
\bigcup_{\lambda\in \sigma(B)}\big\{z\colon|b_\lambda(z)|<\eps\big\}\,,
$$
and therefore it is enough to verify that $f$ is separated from zero on each
disk $\{z\colon|b_\lambda(z)|<\eps\}$, $\lambda\in \sigma(B)$, uniformly with
respect to $\lambda$.

By the Schwarz lemma we have
$$
\left|\frac{f(z)-f(\lambda)}{1-\overline{f(\lambda)}f(z)}\right|\le|b_\lambda(z)|,
$$
i.~e., for all point $z$ of the disk $\{z\colon|b_\lambda(z)|<\eps\}$ we have
$$
\left|\frac{f(z)-f(\lambda)}{1-\overline{f(\lambda)}f(z)}\right|<\eps.
$$
Rewriting the inequality
$$
\left|\frac{a+b}{1+\bar ab}\right|\le\frac{|a|+|b|}{1+|a|\,|b|}
$$
(which is, in fact, the triangle inequality for the hyperbolic metric for the
points $a$, $b$, and $0$) in the form
$$
|b|\ge\frac{|a|-\left|\frac{a+b}{1+\bar ab}\right|}
{1-|a|\left|\frac{a+b}{1+\bar ab}\right|}
$$
with $a=f(\lambda)$ and $b=-f(z)$ we get
$$
|f(z)|\ge\frac{\delta-\eps}{1-\delta\eps}>
\frac{\delta-\delta_1}{1-\delta\delta_1}\,.
$$

Therefore, in the whole half-plane we have
$$
|f(z)|+|B(z)|\ge\eta\,,
$$
where
\begin{equation}
\label{eta}
\eta=\min\left\{\exp\bigg[-\frac{\alpha\sqrt{1+\alpha^2}\left(e^{\pi\alpha}+1\right)}
{2\Big(e^{\frac{\pi\alpha}{\sqrt{1+\alpha^2}+1}}+1\Big)}\;\bigg],\
\frac{\delta-\delta_1}{1-\delta\delta_1}\right\}\,.
\end{equation}

Finally, by the Carleson corona theorem (see, e.g.~\cite{Nik2}), we know that
there exists a solution $h$ of the Bezout equation $fh+Bg=1$ with a norm
estimate
$$
\|h\|_\infty\le\frac c{\eta^2}\log\frac1\eta\,,
$$
which means that
$$
c_1(\delta)\le\frac c{\eta^2}\log\frac1\eta\,.
$$

Recall that $\delta_1=\frac1{\sqrt{1+2\alpha^2}}$. If the first term
in~\eqref{eta} is less than the second one, we have
$$
\eta=\eta(\delta_1)=\exp\bigg[-\frac{\alpha\sqrt{1+\alpha^2}
\left(e^{\pi\alpha}+1\right)}
{2\Big(e^{\frac{\pi\alpha}{\sqrt{1+\alpha^2}+1}}+1\Big)}\;\bigg],
$$
and we can put
$$
C(\delta_1)=\frac c{\eta^2(\delta_1)}\log\frac1{\eta(\delta_1)}\,.
$$
If the second term is smaller, we have
$$
c_1(\delta)\le\frac{c}{(\delta-\delta_1)^2}\log\frac1{\delta-\delta_1}\,.
$$
\end{proof}

\begin{lemma}
\label{l6} Let $\delta\le\delta_1$. Then $c_1(\delta)=+\infty$.
\end{lemma}

\begin{proof}
Consider a sequence of points $v_n$ from Remark~2 and put $f_n=b_{v_n}$. As it
was mentioned in Remark~2,
$$
|b_{v_n}(\lambda)|\ge\frac1{\sqrt{1+2\alpha^2}}=\delta_1\ge\delta\qquad
\forall\lambda\in \sigma(B)\,.
$$
We would like to estimate from below the $H^\infty$ norm of a solution $g_n$ of
the Bezout equation $g_nf_n+B h_n=1$. Since
$$
\|g_n\|_\infty=\|1-Bh_n\|_\infty=\|h_n-\bar B\|_\infty\ge\|h\|_\infty-1
$$
and
$$
\|h_n\|_\infty\ge|h_n(v_n)|=\frac1{|B(v_n)|}\,,
$$
by the estimate of Lemma~\ref{l4} we obtain
$$
\|g_n\|_\infty\to\infty\,,
$$
what yields $c_1(\delta)=+\infty$.
\end{proof}

{\bf Remark 3.} Taking an arbitrary $\delta$, $\delta<\delta_1$, and using the
above construction, it is easy to construct a function $f$ with the properties
$\|f\|_\infty\le1$, $|f(\lambda)|\ge\delta$ for every $\lambda\in\sigma$, which
is not invertible in $H^\infty/BH^\infty$, so that there is no bounded solution
$g,h$ to the Bezout equation $gf+Bh=1$. Indeed, it is sufficient to take for
$f$ a product of the factors $b_{v_n}$ with sufficiently rare subsequence of
zeroes $v_n$ to ensure the condition $|f(\lambda)|\ge\delta$. However, for the
Blaschke product $B$ from Theorem~\ref{main}, we do not know whether there
exists such a function in the case $\delta=\delta_1$. In order to guarantee
this property, i.~e., to have a noninvertible element $f$ of the algebra
$H^\infty/BH^\infty$ with $\delta_1\le f(\lambda)\le\|f\|_\infty\le1$
($\lambda\in\sigma(B)$), we need a Blaschke product $B$ with more sophisticated
zero set, which will be exhibited in the following theorem.

\begin{theorem}
\label{modif} For an arbitrary fixed number $\delta_1$ from $(0,1)$ there
exists a Blaschke product $B$ such that

$1)$ $c_1(\delta,H^\infty/BH^\infty)<\infty$ for every $\delta,$ $\delta_1<\delta\le1;$

$2)$ there exists a function $f$ satisfying
$\delta_1\le|f(\lambda)|\le\|f\|_\infty\le1$ for $\lambda\in\sigma(B),$ but
$\frac1f\notin H^\infty/BH^\infty$.
\end{theorem}

\begin{proof} {\bf Step 1.}
We start with an arbitrary bounded increasing sequence of positive number
$\alpha_n$ with $\alpha=\lim\alpha_n$, $\delta_1\df\frac1{\sqrt{1+2\alpha^2}}$.
Our Blaschke product $B$ will be of the form
$$
B(z)=\prod_{n=1}^\infty\prod_{m=0}^{m_n-1}B_{\alpha_n,\beta_n^m\rho_n}(z)\,,
$$
where
$$
\beta_n=\frac{\sqrt{1+\alpha_n^2}-1}{\sqrt{1+\alpha_n^2}+1}
$$
and
$$\rho_{n+1}=\rho_n\beta_n^{m_n}
\frac{\sqrt{1+\alpha_n^2}+1}{\alpha_n\sqrt{1+\alpha_n^2}}\cdot
\frac{\alpha_{n+1}\sqrt{1+\alpha_{n+1}^2}}{\sqrt{1+\alpha_{n+1}^2}+1}\,.
$$
The initial value $\rho_1=\rho$ can be taken arbitrarily. The claimed
noninvertible function $f$ will be the following Blaschke product
$$
f(z)=\prod_{n=0}^\infty b_{v_n}(z)\,,
$$
where
$$
v_n=\frac1{\rho_{n+1}}\cdot\frac{\alpha_{n+1}\sqrt{1+\alpha_{n+1}^2}}
{\sqrt{1+\alpha_{n+1}^2}+1}\,i=
\frac1{\rho_n\beta_n^{m_n-1}}\cdot\frac{\alpha_n\sqrt{1+\alpha_n^2}}
{\sqrt{1+\alpha_n^2}-1}\,i\,,
$$
i.~e., we put the root $v_n$ on the common boundary of the last strip defined
by $\alpha_n$ and the first strip defined by $\alpha_{n+1}$. So, the only
parameters, which are in our disposition, are the numbers $m_n$ of strips of
equal hyperbolic width or, in other words, the distances between the neighbor
roots $v_n$. We subordinate these distance to the following condition
\begin{equation}
\label{dist}
\Big|\frac{v_l-v_k}{v_l+v_k}\Big|\ge\left(\delta\sqrt{1+2\alpha_{l+1}^2}\right)^{2^{-k}}\,.
\end{equation}

If we take any zero $\lambda$ of the Blaschke product $B$ with $\im
v_{n-1}<\im\lambda<\im v_n$, then
\begin{align*}
|f(\lambda)|&=\prod_{k=0}^\infty|b_{v_k}(\lambda)|=\prod_{k=0}^{n-2}|b_{v_k}(\lambda)|
\cdot|b_{v_{n-1}}(\lambda)b_{v_n}(\lambda)|\cdot\prod_{k=n+1}^\infty|b_{v_k}(\lambda)|
\\
&\ge\prod_{k=0}^{n-2}|b_{v_k}(v_{n-1})|\cdot|b_{v_{n-1}}(\lambda)b_{v_n}(\lambda)|\cdot
\prod_{k=n+1}^\infty|b_{v_k}(v_n)|
\\
&\ge\prod_{k=0}^{n-2}\left(\delta\sqrt{1+2\alpha_n^2}\right)^{2^{-k}}\!\! \cdot
|b_{v_{n-1}}(\lambda)b_{v_n}(\lambda)|\cdot
\prod_{k=n+1}^\infty\left(\delta\sqrt{1+2\alpha_{n+1}^2}\right)^{2^{-k}}
\\
&\ge\left(\delta\sqrt{1+2\alpha_n^2}\right)^{1-3\cdot2^{-n}}\!\!\!\cdot
|b_{v_{n-1}}(\lambda)b_{v_n}(\lambda)|\,.
\end{align*}
Thus, would we guarantee for any root $\lambda$ in the strip between $v_{n-1}$
and $v_n$ the estimate
\begin{equation}
\label{2terms} |b_{v_{n-1}}(\lambda)b_{v_n}(\lambda)|\ge
\frac{\quad\left(\delta\sqrt{1+2\alpha_n^2}\right)^{3\cdot2^{-n}}}{\sqrt{1+2\alpha_n^2}}\,,
\end{equation}
we will immediately obtain the required estimate for $f$:
$|f(\lambda)|\ge\delta$.

{\bf Step 2.} We shall construct the roots $v_n$ by induction. Assume that all
$v_k$ for $k<n$ are already fixed and we need to choose $v_n$. First of all we
have to take $v_n$ far enough from the preceding roots in order to
satisfy~\eqref{dist} for $k=n$ and all $l<n$ as well as for $l=n$ and all
$k<n$.

Note that we need to check condition~\eqref{2terms} only for the roots
$\lambda$ of $B$ with positive real part and the nearest to the imaginary axis,
i.~e., for $\lambda=(1+i\alpha_n)/\rho_n\beta_n^m$, because the hyperbolic
distance between all other $\lambda$ with positive real part and any $v_k$ is
strictly larger, but the consideration for $\lambda$ with negative real part
can be omitted due to the symmetry.

Now, we would like to reduce the problem to the case of two roots  of $B$ only,
the nearest roots to one of the zeroes of $f$, either $v_{n-1}$ or $v_n$,
i.~e., for $m=0$ and $m=m_n-1$.

For the root $\lambda=(1+i\alpha_n)/\rho_n$, we have
$$
|b_{v_{n-1}}(\lambda)|=\frac1{\sqrt{1+2\alpha_n^2}}\,,
$$
and hence~\eqref{2terms} turns into
\begin{equation}
\label{dop1}
|b_{v_n}(\lambda)|\ge\left(\delta\sqrt{1+2\alpha_n^2}\right)^{3\cdot2^{-n}}.
\end{equation}
For the root $\lambda=(1+i\alpha_n)/\rho_n\beta_n^{m_n-1}$ we have
$$
|b_{v_n}(\lambda)|=\frac1{\sqrt{1+2\alpha_n^2}}\,,
$$
therefore~\eqref{2terms} turns into
\begin{equation}
\label{dop2}
|b_{v_{n-1}}(\lambda)|\ge\left(\delta\sqrt{1+2\alpha_n^2}\right)^{3\cdot2^{-n}}\,.
\end{equation}
In fact, both~\eqref{dop1} and~\eqref{dop2} follow from~\eqref{dist}, however
we do not want to enter into these additional estimations and simply
add~\eqref{dop1}--\eqref{dop2} to the list of requirements for the inductive
choice of $m_n$. Now, we check that~\eqref{dop1}--\eqref{dop2} are fulfilled,
as well as~\eqref{dist}, for $m_n$ sufficiently large.

Let us consider the behavior of the function $\phi_a(t)$,
$$
\phi_a(t)\df|b_{ia}\big((1+i\alpha_n)t\big)|^2=
\frac{(a-\alpha_nt)^2+t^2}{(a+\alpha_nt)^2+t^2}\,.
$$
Since
$$
\frac{\phi'_a(t)}{\phi_a(t)}=\frac{4\alpha_na[t^2(1+\alpha_n^2)-a^2]}
{[t^2(1+\alpha_n^2)+a^2]^2-4\alpha_n^2a^2t^2}\,,
$$
the function $\phi_a$ monotonously decreases from 1 to its minimal value, when
$t$ changes from 0 to $a/\sqrt{1+\alpha_n^2}$, and then increases tending again
to 1 as $t\to\infty$. If we consider the product of two such functions
$\phi_a(t)\phi_b(t)$ with sufficiently large $b/a$, then it is clear that this
product has two local minima: first of them tends to $a/\sqrt{1+\alpha_n^2}$
monotonously decreasing as $b\to\infty$, and the second one tends to
$b/\sqrt{1+\alpha_n^2}$ monotonously increasing as $a\to0$.

We apply these arguments to our requirement~\eqref{2terms}. To this end, we set
$$
a=|v_{n-1}|=\frac{\alpha_n\sqrt{1+\alpha_n^2}}{\rho_n(\sqrt{1+\alpha_n^2}+1)}\,,\qquad
b=|v_n|=\frac{\alpha_n\sqrt{1+\alpha_n^2}}{\rho_n\beta_n^{m_n}(\sqrt{1+\alpha_n^2}+1)}\,,
$$
and
$$
\lambda=\frac{1+i\alpha_n}{\rho_n\beta_n^m}\,,
$$
and will compare the values of our function at the points
$t=t_m=1/\rho_n\beta_n^m$, $m=0,\ldots,m_n-1$, in order to guarantee that the
minimal value is attained either for $m=0$ or for $m=m_n-1$, where by our
assumption either~\eqref{dop1} or~\eqref{dop2} is fulfilled. To this aim, we
note that
$$
t_0=\frac1{\rho_n}>\frac{\alpha_n}{\rho_n(\sqrt{1+\alpha_n^2}+1)}=
\frac{a}{\sqrt{1+\alpha_n^2}}\,,
$$
and therefore, for $m_n$ sufficiently large, the point of the minimum is less
then $t_0$ and the function $\phi_a(t)\phi_b(t)$ is increasing at $t_0$.
Symmetrically,
$$
t_{m_n-1}=\frac1{\rho_n\beta_n^{m_n-1}}<
\frac{\alpha_n}{\rho_n\beta_n^{m_n-1}(\sqrt{1+\alpha_n^2}-1)}=
\frac{b}{\sqrt{1+\alpha_n^2}}\,,
$$
and therefore, for $m_n$ sufficiently large, the point of the minimum is bigger
then $t_{m_n-1}$ and the function $\phi_a(t)\phi_b(t)$ is decreasing at
$t_{m_n-1}$. It follows that conditions~\eqref{2terms} is fulfilled for $m_n$
large enough.

Thus, we have proved that conditions~\eqref{dist} and~\eqref{2terms} are
fulfilled if the sequence $m_n$ increases fast enough. Therefore, the
construction of the Blaschke product $B$ and a function $f$ such that
$\|f\|_\infty=1$, $|f(\lambda)|\ge\delta_1$ for all $\lambda$,
$\lambda\in\sigma(B)$, is completed. The function $f$ represents a
noninvertible element of $H^\infty/BH^\infty$. Indeed, if we assume that $f$ is
invertible in $H^\infty/BH^\infty$, i.~e., there exist two $H^\infty$-functions
$g$ and $h$ such that $fg+Bh=1$, then we come to a contradiction, because
$$
\lim_{n\to\infty}\big(f(v_n)g(v_n)+B(v_n)h(v_n)\big)=0\,.
$$

So, we have finished the proof of the second statement of the Theorem.

{\bf Step 3.} In order to complete the proof, we need to check that any other
function $f$ satisfying conditions $\|f\|_\infty=1$, $|f(\lambda)|\ge\delta$
for all $\lambda$, $\lambda\in \sigma(B)$, and for arbitrary $\delta$,
$\delta>\delta_1$, represents an invertible element of $H^\infty/BH^\infty$.

Fix such a $\delta$ and such a function $f$. We need to check that
\begin{equation}
\label{Bezout}
\inf_{\im z>0}(|f(z)|+|B(z)|)>0.
\end{equation}
As in the proof of Theorem~\ref{main}, we take an arbitrary $\eps$,
$\delta_1<\eps<\delta$, and split the upper half-plane in two parts:
$$
\Pi_\eps\df\cup_{\lambda\in \sigma(B)}\{z\colon |b_\lambda(z)|\le\eps\}
$$
and its complement $\Pi_\eps^c$. We will check that $f$ is bounded away from
zero on $\Pi_\eps$ and $B$ does it on $\Pi_\eps^c$.

If $|b_\lambda(z)|\le\eps$, then by Schwarz' lemma
$$
\left|\frac{f(z)-f(\lambda)}{1-\overline{f(\lambda)}f(z)}\right|\le|b_\lambda(z)|\le\eps\,,
$$
and again, as in the proof of Lemma~\ref{l5}, using the triangle inequality in the form
$$
|b|\ge\frac{|a|-\left|\frac{a+b}{1+\bar ab}\right|}
{1-|a|\left|\frac{a+b}{1+\bar ab}\right|}
$$
with $a=f(\lambda)$ and $b=-f(z)$ we get
$$
|f(z)|\ge\frac{\delta-\eps}{1-\delta\eps}
$$
for arbitrary $z$ from $\Pi_\eps$.

On the complement $\Pi_\eps^c$, we estimate $|B(z)|$ splitting the product $B$
into two subproducts $B=B'B''$. Namely, we fix a number $N$ so that
$\delta\sqrt{1+2\alpha_n^2}>1$ for $n\ge N$ and put
$$
B'=\prod_{n=1}^{N-1}\prod_{m=0}^{m_n-1}B_{\alpha_n,\beta_n^m\rho_n}\,, \qquad
B''=\prod_{n=N}^\infty\prod_{m=0}^{m_n-1}B_{\alpha_n,\beta_n^m\rho_n}\,.
$$
Note that the first product is an interpolating Blaschke product. Indeed, all
$B_{\alpha,\gamma}$ are interpolating, because due to the relation
$$
\Big(\frac B{b_\lambda}\Big)(\lambda)=2i\cdot\im\lambda\cdot
\frac{dB}{dz}(\lambda)
$$
we have
$$
B_k(z_k)=\frac{\pi\alpha}{\sinh \pi\alpha}\,,
$$
where $B_k=B_{\alpha,\gamma}/b_{z_k}$, $z_k=(2k+1+i\alpha)/\gamma$. Therefore
zeroes of $B'$ form a finite union of interpolating sets. Since they are
uniformly separated, the whole product $B'$ is interpolating as well. Using a
generalized form of the Carleson condition (see,
e.g.,~\cite{Nik1}--\cite{Nik2})
$$
|B'(z)|\ge c \inf_{\lambda\in \sigma(B')}|b_\lambda(z)|\,,
$$
we get $|B'(z)|\ge c\eps$ in $\Pi_\eps^c$. As to the second product $B''$, we
can use the estimate of Lemma~\ref{l4} for $\rho=\rho_N$. The estimate of
Lemma~\ref{l4} was obtained for the strips of equal hyperbolic width, but in
our situation the width of the strips decreases, because $\alpha_n$ is
increasing. This means that the hyperbolic distance from any point below the
first strip to the corresponding zero of $B_{\alpha_n,\gamma_{n,m}}$ is
strictly bigger than that distance in the case when all $\alpha_n$ are equal to
$\alpha$. Therefore, below the first strip, each factor
$|B_{\alpha_n,\gamma_{n,m}}|$ is strictly larger than in the equidistant case.
The whole half-plane
$$
\im z>\im v_{N-1}=\frac1{\rho_N}\cdot\frac{\alpha_N\sqrt{1+\alpha_N^2}}
{\sqrt{1+\alpha_N^2}+1}
$$
is in the set $\Pi_\eps$ by Lemma~\ref{l3}, therefore $B''$ is separated from
zero in the set $\Pi_\eps^c$, whence the whole $B$ is separated from zero on
$\Pi_\eps^c$. So, condition~\eqref{Bezout} is fulfilled what means the
invertibility of $f$ in the algebra $H^\infty/BH^\infty$, i.~e.,
$c(\delta)<\infty$ for any $\delta$, $\delta>1/\sqrt{1+2\alpha^2}$.
\end{proof}

\section{A version of the restricted invertibility conjecture}

\subsection{Bourgain--Tzafriri's restricted
invertibility theorem}

The following statement is known as {\it Bourgain--Tzafriri's restricted
invertibility theorem}.

\begin{theorem} \textup(\cite{BTz}\textup)
Whatever are a bounded operator $T$ on a Hilbert space $H$ and an orthogonal
basis $\{e_j\}_{j\in\N}$ satisfying $\inf_j\frac{\|Te_j\|}{\|e_j\|}>0$\textup,
there exists a subset $I\subset\N$ of positive upper density
$$
0<\overline\dens(I)\df\limsup_{n\to\infty}\frac{| I\cap \{1,2,\ldots,n\}|}n
$$
such that the restriction $T|H_I$ is left invertible\textup:
$$
\inf\{\|Tx\|\colon x\in H_I,\|x\|=1\}>0\,,
$$
where $H_I=\spn\{e_j\colon j\in I\}$.
\end{theorem}

See also~\cite{SS} for a generalization and a simpler proof of a matrix version
of Bourgain--Tzafriri's Theorem.

The following conjecture often is quoted as Bourgain--Tzafriri's {\it
restricted invertibility conjecture}~({\bf RIC}) (it seems although that these
authors never actually stated this as a conjecture). It is known that the
famous Kadison--Singer conjecture on pure states on $C^*$-algebras
(see~\cite{KS}, \cite{CT}) implies {\bf RIC}: it is proved in~\cite{CCLV} that
if Kadison--Singer problem has a positive solution then the {\bf RIC} has as
well. For more details about these conjectures we refer to the papers mentioned
above, as well as to a WEB page~\cite{ARCC}. Both conjectures are still open
(June 2010).

\subsection{Restricted Invertibility Conjecture (RIC)}

\begin{conjecture}
For every bounded operator $T$ on a Hilbert space $H$ and every orthogonal
basis $\{e_j\}_{j\in\N}$ satisfying $\inf_j\frac{\|Te_j\|}{\|e_j\|}>0$\textup,
there exists a finite partition $\bigcup_{s=1}^r\!I_s=\N$ such that all
restrictions $T|H_{I_s}$ are left invertible.
\end{conjecture}

It is easy to see that the {\bf RIC} is equivalent to require the same quality
partitions for every bounded $T$ and every {\it unconditional} basis in $H$ (in
place of orthogonal ones). Knowing no much progress in this conjecture during
the last 20 years, we can try to approach the truth treating first some
stronger conjectures. Namely, we can replace here an ``unconditional basis'' by
a ``Schauder basis'', and even by a ``summation basis''. {\it We denote the
corresponding conjectures by {\bf B-RIC} and {\bf SB-RIC}, respectively}.

Precisely, a {\it summation basis} relative to a (triangular) matrix
$V=\{v_{nj}\}$ of scalars $v_{nj}$ is a sequence $\{e_j\}_{j\in\N}$ in $H$ such
that for every $x\in H$ there exists a unique sequence of scalars
$\{a_j\}_{j\in\N}$ satisfying $x=(V)\sum_{j\ge1}a_je_j$, which means the
following:
\begin{itemize}
\item{$v_{nj}=0$ for $j>n$;}
\item{$x=\lim_{n\to\infty}\sum_{j=1}^nv_{nj}a_je_j=x$ (norm convergence).}
\end{itemize}

Clearly, {\bf SB-RIC} $\implies$ {\bf B-RIC} $\implies$ {\bf RIC}. Here, we
present a counterexample to the {\bf SB-RIC}.

\subsection{Counterexample}
\begin{theorem}
\label{counterexample}
Given $\delta,$ $0<\delta<1,$ there exists a sequence
$\{e_j\}_{j\in\N}$ in a Hilbert space $H$ satisfying the following properties.
\begin{itemize}
\item[(1)] {$\{e_j\}_{j\in\N}$ is a summation basis \textup(relative to a triangular matrix\textup).}

\item[(2)] {$\{e_j\}_{j\in\N}$ is block orthogonal\textup: there exists an
increasing sequence of integers $n_s$ such that $H_{[n_s,n_{s+1})}\perp
H_{[n_t,n_{t+1})}$ for every $s\ne t,$ where $H_{[n_s,n_{s+1})}=\spn\{e_j\colon
n_s\le j<n_{s+1}\}$.}

\item[(3)] {There exists a bounded operator $A\colon H\to H$ satisfying
$\|A\|\le1,$ $Ae_j=\lambda_j(A)e_j,$
$\delta\le|\lambda_j(A)|=\frac{\|Ae_j\|}{\|e_j\|}\le\|A\|\le1$ $(j\in\N),$ and
such that for every finite partition $\bigcup_{s=1}^rI_s=\N$ there is a
restriction $A|H_{I_s}$ $(1\le s\le r),$ which is NOT left invertible.}

\item[(4)] {Every bounded operator $T\colon H\to H$ satisfying
$Te_j=\lambda_j(T)e_j$ \textup($j\in\N$\textup) and
$1\ge\|T\|\ge\inf_j|\lambda_j(T)|>\delta$ is invertible.}
\end{itemize}
\end{theorem}

\begin{proof}
We use our main construction from Theorem~\ref{modif} replacing the upper
half-plane by the unit disk. Namely, given $\delta$ ($\delta_1$ in the
Theorem), $0<\delta<1$, there exists a Blaschke sequence $\sigma=\{z_j\}$ of
distinct points in $\D$ such that

\begin{itemize}
\item[a)] {for every $f\in H^\infty$ with $\delta<\inf_j|f(z_j)|$ and
$\|f\|_\infty\le1$, we have $\frac1f\in H^\infty|\sigma$;}

\item[b)] {there is an $H^\infty$ function $g$ such that $\delta\le|
g(z_j)|\le\|g\|_\infty\le1$ but $\frac1g\not\in H^\infty|\sigma $.}
\end{itemize}

Now, we interpret~a) and~b) in terms of the model operator $M_B^*$ and the
reproducing kernels $x_j=\frac{(1-|z_j|^2)^{1/2}}{1-\bar z_jz}$.

First, by (a scalar version of) the commutant lifting theorem, an operator $T$
from point~(4) of the Theorem is of the form $T=f(M_B)^*$, where $f$ satisfies
all properties from~a): $f\in H^\infty$, $\delta <\inf_j|f(z_j)|$ and
$\|f\|_\infty\le1$. Hence, $\frac1f\in H^\infty|\sigma$, which means that $T$
is invertible (and proves point~(4)).

Secondly, in order to fix statements~(1)--(3), we restate item~b) above in
terms of the same model operator. Namely, for an operator
$$
T=g(M_B)^*
$$
with a function $g$ from~b) we have $\|T\|\le1$, $Tx_j=\lambda_j(T)x_j$,
$\delta\le|\lambda_j(T)|=|g(z_j)|\le\|T\|\le1$ ($j\in\N$), and
$\inf\{\|Tx\|\colon x\in K_B,\|x\|=1\}=0$.

Notice that if we would like to restrict ourselves to properties~(1) and~(3)
only, we simply set $A=T=f(M_B)^*$. Property~(1) follows from the fact that the
sequence $\{x_j\}_{j\in\N}$ corresponding to a Blaschke sequence
$\{z_j\}_{j\in\N}$ is a summation basis,~\cite{Nik1}, p.~194. In order to
check~(3), suppose that there exists a partition $\bigcup_{s=1}^r\sigma_s=\N$
such that all restrictions $T|H_{\sigma_s}$ are left invertible:
$$
0<\inf\{\|Tx\|\colon x\in H_{\sigma_s}, \|x\|=1\}
$$
for every $s$, $1\le s\le r$. We lead this to contradiction as follows. Let
$B_s$ be the Blaschke product whose zero sequence is $\sigma_s$. Since the
restriction $T|H_{\sigma_s}=g(M_{B_s})^*$ is, in fact, invertible, there exist
functions $f_s,h_s\in H^\infty$ such that $gf_s+B_sh_s=1$. Hence,
$B\cdot\prod_{s=1}^rh_s=\prod_{s=1}^r(1-gf_s)=1-gF$, where $ F\in H^\infty$.
This shows that the operator $T=g(M_B)^*$ is invertible, what contradicts the
construction of $T$. Therefore, a counterexample satisfying properties (1),
(3), and~(4) of the Theorem is constructed.

In order to satisfy property~(2), we modify the previous construction in the
following way. Let $N\in\N$ and $T_N=T|H_N$ be the restriction of $T$ to
$$
H_N=\spn\{x_j\colon1\le j\le N\}\,.
$$
Then
\begin{gather*}
\|T_N\|\le1,\quad T_Nx_j=\lambda_j(T_N)x_j\,,
\\
\delta\le|\lambda_j(T_N)|\le \|T\|\le1\quad (1\le j\le N)\,,
\\
\intertext{and} \lim_{N\to\infty}\inf\{\|T_Nx\|\colon x\in H_N,\|x\|=1\}=0\,.
\end{gather*}
Now, we set
$$
A=\sum_{N\ge1}\oplus T_N\,,
$$
which is defined coordinate-wise on an ($l^2$) orthogonal sum
$$
H=\sum_{N\ge1}\oplus H_N\,.
$$
In particular, this means that the point spectrum of $A$ is
$\{\lambda_j(T)\}_{j\ge1}$ but each eigenvalues is repeated infinitely many
times.

Next, we denote $\{e_j\}_{j\in\N}$ the sequence of eigenvectors of $A$ ordered
naturally: if $f_k=(\delta_{N,k})_{N\ge1}\in l^2$, then
$$
\{e_j\}_{j\in\N}=(x_1f_1,x_1f_2,x_2f_2,\dots,x_1f_N,x_2f_N, \dots,x_Nf_N,x_1f_{N+1},\dots)\,,
$$
or, more formally,
$$
e_j=x_mf_N\,,\qquad\text{where}\quad N=[\sqrt{2j}+\frac12],\quad m=j-\frac{N(N-1)}2\,.
$$

Show that $\{e_j\}_{j\in\N}$ satisfies properties~(1) and~(2), and $A$ fulfils all requirements of~(3).

Indeed, properties~(1) and~(2) for $\{e_j\}$ easily follow from the
property~(1) for $\{x_j\}$ and a block orthogonal nature of $\{e_j\}$.

In order to prove~(3), suppose the contrary, i.~e., that there exists a finite
partition $\bigcup_{s=1}^rI_s=\N$ such that all restrictions $A|H_{I_s}$ ($1\le
s\le r$) are left invertible. Taking an intersection of $\bigcup_{s=1}^rI_s=\N$
with the \hbox{$N$-th} group of indices corresponding to the eigenfunctions
$\{x_mf_N\}_{1\le m\le N}$, we obtain a partition $\bigcup_{s=1}^rI_{s,N}=I^N$
of the set $I^N=\{1,2,\dots,N\}$, where index $m$ runs. Reasoning by induction,
assume we have an infinite subsequence $\{N_i\}$ of $\N$ such that for a given
$N$ all partitions $\bigcup_{s=1}^r(I_{s,N_i}\bigcap I^N)=I^N$, $i\ge1$, are
the same. Since there is only a finite number of partitions of $I^{N+1}$, we
can choose an infinite subsequence of $\{N_i\}$, say $\{N'_l\}$, such that all
partitions $\bigcup_{s=1}^r(I_{s,N'_l}\bigcap I^{N+1})=I^{N+1}$, $l\ge1$, are
the same. Applying a diagonal process to this table of sequences, we obtain a
growing sequence of integers $\{M_i\}_{i\ge1}$ such that all partitions
$\bigcup_{s=1}^r(I_{s,M_i}\bigcap I^N)=I^N$, $i\ge1$, are the same, for all
$N=1,2,\dots$. This means that we have a partition
$\bigcup_{s=1}^r\sigma_s=\N$, $I_{s,M_i}\bigcap I^N=\sigma_s\bigcap I^N$. Next,
we observe that, for every $s$, $1\le s\le r$,
\begin{align*}
0<\delta:=&\inf\{\|Ax\|\colon x\in H_{I_s},\|x\|=1\}
\\
\le&\inf\{\|T_Nf\|\colon f\in H_{I_{s,M_i}\cap I^N},\|f\|=1\}
\end{align*}
for every $N\ge1$. Taking $N\to\infty$, we get
$$
0<\delta\le\inf\{\|Tf\|\colon f\in H_{\sigma_s},\|f\|=1\}
$$
for every $s$, $1\le s\le r$.

But, as we saw above, this is impossible.
\end{proof}

\end{document}